\documentclass[12pt]{amsart}
\usepackage[utf8]{inputenc}
\usepackage[english]{babel}
\usepackage{amsmath}
\usepackage{amsthm}
\usepackage{amssymb}
\usepackage{graphicx}
\newtheorem{lemma}{Lemma}[section]
\newtheorem{proposition}{Proposition}[section]
\newtheorem{theorem}{Theorem}[section]
\newtheorem{corollary}{Corollary}[section]
\newtheorem{algo}{Algorithm}
\theoremstyle{definition}
\newtheorem{definition}{Definition}
\newtheorem{example}{Example}[section]
\usepackage[
backend=biber,
style=alphabetic,
sorting=ynt
]{biblatex}
\begin{document}
\title{A Generalization of the Greene-Kleitman Duality Theorem}
\author{Frank Y. Lu}
\date{\today}
\address{Department of Mathematics, Princeton University, Princeton, NJ 08544, USA}
\email{lfrank2015y@gmail.com}

\maketitle
\textbf{Abstract:} In this paper, we describe and prove a generalization of both the classical Greene-Kleitman duality theorem for posets and the local version proved recently by Lewis-Lyu-Pylyavskyy-Sen in studying discrete solitons, using an approach more closely linked to the approach of the classical case. 

\section{Introduction}
The Greene-Kleitman duality theorem for finite posets, first described in Greene's paper, \cite{GREENE197669} (see also \cite{britz1999finite}, upon which the following exposition is loosely based) states the following result. Given a poset $P,$ let $A_k$ be the maximal possible sum of the lengths of $k$ disjoint increasing sequences of elements (chains), and $D_k$ is the maximal possible sum of the lengths of $k$ disjoint sequences of elements where no two elements are pairwise comparable (anti-chains). Then $A_k, D_k$ are conjugate in the following sense: $A_1 + (A_2 - A_1) + \cdots$ and $D_1 + (D_2 - D_1) + \cdots$ form conjugate partitions of $n.$ See \cite{britz1999finite}[\S 8] for a description of one proof of this result, attributed to A. Frank, using a graph theoretic construction, which will be relevant to us. Here, we will refer to this as the classical Greene-Kleitman duality theorem.
\par The duality of these partitions lends itself to applications, which \cite{britz1999finite} discusses in detail. For instance, \cite{britz1999finite}[\S 3-4] goes into how one can interpret results about tableau associated with permutations through this lens of the duality. This is done by using the above theorem on the permutation poset, or the poset associated with a given permutation $\sigma$ of $n$ elements by imposing the ordering, on the set of elements $\{(i, \sigma(i))|i = 1, 2, \ldots, n\},$ that $(i, \sigma(i)) < (j, \sigma(j))$ if and only if $i < j$ and $\sigma(i) < \sigma(j).$
\par Recently, another related duality result was published in \cite{lewis2020scaling}, which \cite{Lewis2019} (from which the exposition regarding this theorem is based) calls the \textit{localized Greene's theorem}. Here, we start with a permutation $\sigma$ on a set of elements $\{1, 2, \ldots, n\},$ and consider the sequence $\sigma(1), \sigma(2), \ldots, \sigma(n).$ From there, the same sort of duality in the original duality theorem was shown: however, instead of the original values for the classical theorem being conjugate, we have quantities $A'_k$ and $D'_k$ being conjugate, which are defined as follows. For $A'_k,$ we consider, over all sets of $k$ disjoint subsequences, the maximum of the sum of the \textit{ascents} of each subsequence, where the ascent of a subsequence $s_1, s_2, \ldots s_m$ is the number of indices $i$ so that $s_i < s_{i+1},$ plus one (or $0$ if the sequence is empty). For $D'_k,$ this is defined to be the maximum, over all sets of $k$ consecutive subsequences, of the sum of the lengths of the maximal descending sequences of each subsequence. Note the consecutive condition here, in contrast with the classical theorem: for instance, if we have $2, 4, 3, 1$ as the sequence, we may not take the sequences to be $2, 1$ and $4, 3.$ The proof in \cite{lewis2020scaling} of this result differs substantially from the classical proof, utilizing the study of discrete solitons in the paper.
\par In this paper, we unite these two theorems with a generalization, Theorem 2.1, which we detail in the next section, using an overall structure of the proof similar to the proof provided by Frank. Here, we translate the problem into a problem about direct graphs and flows on direct graphs; again, see \cite{britz1999finite}[\S 8] for one version of Frank's proof, for instance, upon which the main ideas for this proof are built. However, the structure of the graph is constructed in a way where flows and potentials correspond more ``naturally" to sequences. 
\par Specifically, in Section 2, we introduce the generalized theorem and the specific information necessary. From there, in section 3 we set up the required graph theory to allow for the translation of the problem to this graph theoretic construction, along the lines of the classical proof. From here, we prove some basic properties of the graph construction which will be useful in Section 4, before proceeding on to the core part of the proof. In Section 5, we link the two desired poset-based quantities to the graph-theoretic construction and use this to arrive at an inequality, which we sharpen to the desired equality in Section 6. Finally, in Section 7, we show that both versions of the Greene-Kleitman duality theorem follow as corollaries of this general theorem, and provide another interesting special case.
\par Thanks to Dr. Pavlo Pylyavskyy for introducing me to this problem, as well as offering suggestions on drafts, including on the exposition in sections 1, 2, and 3, and the abstract. Thanks as well to Dr. Emily Gunawan for suggestions on the draft, especially with regards to the exposition of sections 2 and 3, including the example, and thanks to Dr. Joel Lewis for comments on the earlier version of the draft, in particular on the exposition in sections 2 and 3 as well.
\section{The Generalized Problem}
The exposition loosely adapts from \cite{Lewis2019} in generalizing this problem, in that the notation and exposition here generalizes that of the localized Greene's theorem given in \cite{Lewis2019}.
\par Given a poset $P$ on elements $S_P = \{e_1, e_2, \ldots, e_n\}$ and a bijection $h: S_P \rightarrow \{1, 2, \ldots, n\},$ we pick a set $C_P$ of pairs of distinct elements in $S_P$ with the following properties:
\begin{enumerate}
    \item Given $x, y \in S_P,$ if $x < y$ and $h(x) < h(y),$ then $(x, y) \in C_P.$
    \item Given that $(x, y) \in C_P,$ we have that $h(x) < h(y).$
    \item Given that $(x, y) \in C_P$ and $(y, z) \in C_P,$ we have that $(x, z) \in C_P.$
\end{enumerate}
\par In other words, $C_P$ is some binary transitive relation on $S_P$ that is a subset of the strict total ordering given by $h,$ which also contains the intersection of the relations given by $P$ and the relations given by $h.$
\par Given the set $C_P$ and bijection $h,$ we say that a sequence of distinct elements $s_1, s_2, \ldots, s_m$ is \textit{adjacentable} if for each $j, 1 \leq j \leq m-1,$  $(s_j, s_{j+1}) \in C_P.$ In addition, we say that the sequence is $h-$\textit{ordered} if it satisfies that $h(s_j) < h(s_{j+1})$ for each $j.$ Note that adjacentable sequences are necessarily $h-$ordered, but the reverse isn't true if $C_P$ is a strictly smaller relation than $h.$
\par Now, let $S$ be an adjacentable sequence of distinct elements $s_1, s_2, \ldots, s_m.$ Define $asc(S)$ to be the number of indices $j$ so that $s_j < s_{j+1},$ plus one, or to equal $0$ if the sequence is empty. 
\par In addition, for any $h-$ordered sequence $S$ of distinct elements, define $desc(S)$ to be the length of the longest subsequence of $S,$ say $s_1, s_2, \ldots, s_n,$ so that $s_i \not < s_j$ for each $i < j.$ 
\begin{example} Suppose we have a poset $P$ on the set $\{a, b, c, d, e\},$ with the cover relations $a<b, b<d, c<d, d<e,$ and the function $h$ that takes on the following values:
\begin{align*}
    h(a) = 1 \\ h(b) = 3 \\ h(c) = 5 \\ h(d) = 4 \\ h(e) = 2.
\end{align*} Let $C_P$ be the set $\{(x, y) \in \{a, b, d, e\} \times \{a, b, d, e\}| h(x) < h(y)\} \cup \{(a, c)\}$ If we have the sequence $S$ be $(a, e, b, d),$ we have $asc(S) = 3,$ as $a<e$ and $b<d.$ Also $desc(S) = 2,$ by taking the subsequence $e, d.$ Note that this will naturally be $0$ if $S$ is empty.
\end{example}
\par We say that two disjoint $h-$ordered sequences $s_1, s_2, \ldots, s_m$ and $t_1, t_2, \ldots, t_l$ of $P$ are \textit{semi-overlapping} if and only if there exist indices $i, j, k, l$ so that $(t_j, s_i)$ and $(s_k, t_l)$ lie in $C_P.$ For instance, note that $a, e, d$ and $b, c$ are semi-overlapping (since $f(d) > f(b), f(a) < f(b)$), but $a, e$ and $b, c$ aren't. 
\par From here, define $A_k'$ to be the maximum value, over all sets of $k$ disjoint adjacentable sequences $\{S_1, S_2, \ldots, S_k\}$ of $asc(S_1) + asc(S_2) + \cdots + asc(S_k).$ Similarly, define $D_k'$ to be the maximum value, over all sets $\{S_1, S_2, \ldots, S_k\},$ of $k$ disjoint $h-$ordered sequences where no two are semi-overlapping, of $desc(S_1) + desc(S_2) + \cdots + desc(S_k).$ 
\par For Example 2.1, we compute that $A_1' = 3,$ using the sequence $S = (a, e, b, d).$ Similarly, we see that $A_2' = 4,$ using $S_1 = (a, e, b, d)$ and $S_2 = (c),$ and $A_3' = 5,$ using $S_1 = (a, e),$ $S_2 = (b, d)$ and $S_3 = (c).$ We compute also that $D_1' = 3,$ using the sequence $(e, b, c),$ $D_2' = 4$ using the sequences $(e, b, c)$ and $(d),$ and $D_3' = 5$ using the sequences $(a), (e, b, c),$ and $(d).$
\par Given these quantities, we have the following theorem.
\begin{theorem}
Let $\lambda_1 = A_1',$ and $\mu_1 = D_1',$ and for $k \geq 2,$ let $\lambda_k = A_k' - A_{k-1}', \mu_k = D_k' - D_{k-1}'.$ Then, the sums $n = \lambda_1 + \lambda_2 + \cdots$ and $n = \mu_1 + \mu_2 + \cdots$ are partitions; moreover, they are conjugate partitions. 
\end{theorem}
\par For instance, as we'll show in Section 7, if we let $P$ be the natural ordering on the set of elements $\{1, 2, \ldots, n\},$ if we take $C_P$ to be the set $\{(x, y) \in S_P \times S_P|h(x) < h(y)\}$ and $h$ to be the permutation, we will arrive at the localized Greene's theorem for permutations from \cite{lewis2020scaling}. Also, if we let $P$ be a poset, $h$ to be a linear extension, and $C_P = \{(x, y)| x < y\},$ we will arrive at the Greene-Kleitman theorem from \cite{GREENE197669}. This latter result, however, will require a little bit more work, as we will do in Section 7.
\par As mentioned before, the general method of proof is similar to \cite{britz1999finite} \S 7 and \S 8, which was used to prove the classical Greene-Kleitman theorem. 
\section{Setup}
In this section, we establish a directed graph which reflects the structure of the poset $P.$ The exposition in this section follows \cite{britz1999finite} \S 7, with modifications to the theorems in the section, though we will also borrow some exposition from \cite{williamson_2019} when needed for modifications.
\subsection{The Graph}
\par Given a poset $P$ on set $S_P$ with $n$ elements, and bijection $h$ between elements of $P$ and the set $\{1, 2, \ldots, n\},$ we now construct a directed graph $G_{P, h, C_P} = (V, E).$ Here, the set $V$ consists of $2n + 2$ elements: a source vertex $b_0,$ a sink vertex $t_{n+1},$ and for each element $e \in P,$ we have a ``top" vertex $t_{h(e)}$ and a ``bottom" vertex $b_{h(e)}.$ The set of edges $E$ is the union of the following four sets, where we have the ordered pair $(v, w)$ represent a directed edge from vertex $v$ to vertex $w:$
\begin{enumerate}
    \item The set $\{(b_0, t_i)| 1 \leq i \leq n\}$ of edges from $b_0$ to each of the vertices $t_i.$ 
    \item The set $\{(b_i, t_{n+1})| 1 \leq i \leq n\}$ from each of the $b_i$ to $t_{n+1}.$
    \item The set $\{(t_i, b_i)| 1 \leq i \leq n\}$ from each $t_i$ to its corresponding $b_i.$ 
    \item The set $\{(b_i, t_j)| (h^{-1}(i), h^{-1}(j)) \in C_P\}.$
\end{enumerate}
\par Notice that these four sets of edges are distinct. For Example 2.1, we get a graph like the following graph. 
\begin{figure}[h]
    \centering
    \includegraphics[scale = 1.2]{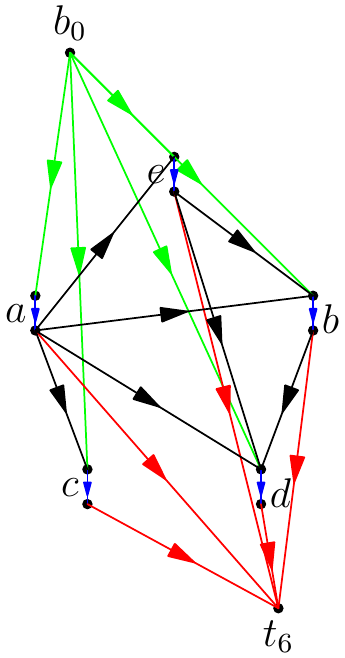}
    \caption{$G_{P, h, C_P}$ for Example 2.1}
    \label{fig:my_label}
\end{figure}
\par Here, green represents the first set, red represents the second set, blue represent the third, and black represent the fourth set. Next to each pair of vertices $t_i, b_i$ for $i$ from $1$ to $5$ is the element in $P$ that it corresponds to (namely, $h^{-1}(i)$).
\subsection{Minimal-Cost Flow}
We now consider imposing a flow onto the graph, and finding, for a given flow value $v$ (defined as in \cite{britz1999finite} as the sum of the flows assigned to each edges going out of a source node), the minimal cost flow. As mentioned before, the exposition we use here is similar to that of \cite{britz1999finite} \S 7, with some adaptations from \cite{williamson_2019}.
\par We use the definition of flow used in \cite{britz1999finite} \S 7: a \textit{flow} on a directed graph with vertex set $V$ and edge set $E,$ with one source node and one sink node, is a function $f: E \rightarrow \mathbb{R}_{\geq 0}$ so that, for each vertex $v$ that isn't a source or a sink, $\sum\limits_{(w, v) \in E} f((w, v)) = \sum\limits_{(v, w) \in E} f((v,w)).$ This property is also known as flow conservation. The value of a flow is then just $\sum\limits_{(s, w) \in E} f((s, w)),$ where $s$ is the source node. Notice that this flow can be restricted in value; the \textit{capacity} of a given edge gives us the bounds for what values $f$ can take on the edge. For this discussion, as $f$ is nonnegative, we let the capacity function simply be the maximum value that $f$ can take on each edge.
\par Now, for the costs of this graph, define the function $c: E \rightarrow \mathbb{Z},$ so that an edge $e = (v, w) \in E$ has cost $-1$ if $w = t_{n+1},$ or $v = b_i, w = t_j,$ where $h^{-1}(i) < h^{-1}(j)$ and $i < j,$ and all other edges have cost $0.$ Define as well the capacity function $u: E \rightarrow \mathbb{Z}$ that sets the capacity of all edges to be $1.$
\par The exposition from here follows that of \cite{williamson_2019}, as \cite{britz1999finite} doesn't provide us with a sufficiently general context, though we will return to the mechanics of \cite{britz1999finite} afterwards.
\par Working more directly in the context of \cite{williamson_2019}, we have the following definition:
\begin{definition}[Definition 5.2 from \cite{williamson_2019}]
Given a directed graph with vertices $V$ and edges $E,$ we add to the edges the reverse of these edges (so if $(v, w) \in E,$ we add $(w, v)$), and we denote the total set of edges as $E'.$ Suppose we are also given a function $u: E' \rightarrow \mathbb{Z}$ with $u(e) \geq 0$ for all $e \in E',$ and cost function $c: E' \rightarrow \mathbb{Z},$ where $c((v, w)) = -c((w, v)).$ Then, a circulation is a function $g: E' \rightarrow \mathbb{R}_{\geq 0}$ is a function satisfying the following properties: 
\begin{itemize}
    \item For all edges $e \in E',$ we have that $g(e) \leq u(e).$
    \item For all vertices $i \in V,$ we have that $\sum\limits_{k \in V|(i, k) \in E} g((i, k)) = 0.$ 
    \item For all vertices $v, w$ so that $(v, w) \in E',$ $g((v, w)) = -g((w, v)).$
\end{itemize}
We say that the cost of the circulation is $\frac{1}{2} \sum\limits_{e \in E} c(e)g(e),$ which we denote as $c(g).$
\end{definition}
\par We have the following theorem from \cite{williamson_2019} which corresponds to \cite{britz1999finite}[Theorem 7.1], giving us certain criteria for when we have the minimal cost flow. We weaken the theorem to only the needed conditions.
\begin{theorem}[part of Theorem $5.3$ from \cite{williamson_2019}]
The following are equivalent for a circulation $g,$ given capacity and cost functions $u, c$ respectively:
\begin{itemize}
    \item $g$ is a minimum-cost circulation.
    \item There exists a potential function $p: V \rightarrow \mathbb{R}$ so that for all vertices $v,w$ where $(v, w) \in E'$ and $u((v, w)) - g((v,w)) \geq 0,$ $c((v, w)) + p(v) - p(w) \geq 0.$
\end{itemize}
\end{theorem}
\par Using this theorem, we prove that a strengthened version of \cite{britz1999finite}[Theorem 7.1] holds. Suppose that we have a flow $f$ and potential $p$ on $G_{P, h, C_P},$ with the cost function $c$ and capacity function $u,$ so that $f$ always lies between $0$ and $u$ for each edge in $E.$ We prove the following theorem.
\begin{theorem}[Modified Theorem 7.1 from \cite{britz1999finite}]
Let $G$ be a directed graph, with set of vertices $V$ and set of edges $E,$ with a single source and a single sink vertex. If we have a flow $f$ and potential $p$ so that $$p(w) - p(v) < c((v, w)) \implies f((v, w)) = 0,$$ and $$p(w) - p(v) > c((v, w)) \implies f((v, w)) = u((v, w))$$ for any $(v, w) \in E,$ then $f$ has minimal cost over all flows of the same value; that is, the sum of $f(e)c(e)$ over all edges $e$ is minimal for this flow.
\end{theorem}
\begin{proof}
\par Suppose that the flow $f$ satisfies these conditions, with value $v.$ We'll show that it is minimal by comparison with \cite{williamson_2019}[Theorem 5.3]. Denote the source node $a$ and the sink node $b.$
\par First, as in \cite{williamson_2019}[\S 5], given the graph $G = (V, E)$ and a desired flow value $v,$ add to $G$ the vertex $s,$ and two edges, one from $s$ to $a,$ and one from $b$ to $s,$ both with capacity $v$ and cost $0.$ Given $p$ as well, extend the potential function so that $p(s) = p(a)$ as well.
\par In addition, once we've added these two edges, perform the modifications in the beginning of definition $1.$ Specifically, let $E'$ be the new set of edges. Extend the capacity function $u$ to $u': E' \rightarrow \mathbb{R}$ that is $v$ on the edges $(s, a)$ and $(b, s),$ $-v$ on their reverses, and $0$ on all the other edges not in $E$. Furthermore, extend the cost function to equal $0$ on the new edges.
\par Now, \cite{williamson_2019} notes that given this flow, there is a corresponding circulation with the same cost. We show this more precisely.
\par To do this, given any flow $f',$ construct a function $g'$ where the following hold: 
\begin{center}
  $g((v, w)) = f((v,w))$ for all $(v, w)$ in $E.$
  \\ $g((v, w)) = -f((w, v))$ for all $(v, w)$ in $E' - E.$
  \\ $g((s, a)) = g((b, s)) = v.$
  \\ $g((a, s)) = g((s, b)) = -v.$
\end{center} It is not hard to check that this is a circulation. 
\par By construction, notice that the cost of $g'$ and the cost of $f'$ are the same. Notice that the two new edges and their respective ``reversed" edge have cost $0$ and so don't contribute to the total cost. Also, observe that for every edge $(v, w) \in E$ in the circulation, the contribution of the cost due to $(v, w)$ and $(w, v)$ in total is $\frac{1}{2} (c((v, w))f((v, w)) + c((w,v))f((w, v))) = c((v, w))f((v, w)),$ and summing these up yields the same cost.
\par Now, let $g$ be the circulation constructed from the particular flow $f$ mentioned at the beginning of the proof. Notice that for all $(v, w) \in E',$ we have three cases to consider.
\begin{enumerate}
    \item First, if $(v, w) \in E,$ by construction notice that if $u((v, w)) > f((v, w)),$ then by construction we see that $p(w) - p(v) \leq c((v, w)),$ or that $c((v, w)) - p(w) + p(v) \geq 0.$
    \item Next, if $(v, w)$ is so that $(w, v) \in E,$ then notice that $u((w, v)) > f((w, v)) \iff f((v, w)) > 0,$ which in turn means that $p(w) - p(v) \geq c((v, w)),$ or that $c((w, v)) - p(v) + p(w) \geq 0.$
    \item For the last four edges, notice that their circulation equals their capacity, so there's nothing that needs to be checked here.
\end{enumerate}
Then, it follows that $g$ is a minimal cost circulation, which means that the cost of $g$ is at most the cost of $g'.$ But then it follows that the cost of $f$ is at most the cost of $f',$ for any flow $f'$ with value $v,$ which gives us the desired.
\end{proof}
\par In particular, notice that the first condition in \cite{britz1999finite}, namely that the potential is bounded between its values at the source and sink nodes, is not necessary to maintain for minimality, thus allowing us more flexibility with the potential function. We are now able to return back to the notation of \cite{britz1999finite}, but now with the possibility of negative potentials and costs.
\subsection{Applying the Algorithm}
\par We now apply \cite{britz1999finite}[Algorithm 7.2] to the graph $G_{P, h, C_P},$ with the $2n + 2$ vertices $b_0, t_1, b_1, \ldots, t_{n+1},$ but with a few modifications (specifically to the initial conditions), which are produced below. Let $V$ be the set of vertices and $E$ the set of edges in this graph.
\begin{algo}[Modified Algorithm 7.2 from \cite{britz1999finite}]
The algorithm is as follows:
\begin{enumerate}
    \item To initialize the flow and potential, set $f$ to be so that $f(e) = 0$ for every edge $e \in E.$ We also declare that $p(b_i) = -i = p(t_i).$
    \item Let $G'$ be the modified graph with the same vertices and edges $\bar{E} = \{(v, w): (v, w) \in E, p(w) - p(v) = c((v, w)), f((v, w)) < u((v, w))\} \cup \{(w, v): (v, w) \in E, p(w) - p(v) = c((v, w)), f((v, w)) > 0\}.$ From here, let $X$ be the set of vertices $v$ where a path exists from the source $s$ to $v$ using the edges in $\bar{E}.$ If $t \in X,$ then go to step 3. Otherwise, go to step 4.
    \item There exists a path through vertices $s, v_1, v_2, \ldots, v_k, t,$ where all these vertices lie in $X$ and the edges are in $\bar{E}.$ Increase the flow of each edge along here by $1,$ then go to step 5.
    \item Otherwise, for every vertex not in $X,$ increase the potential of that vertex by $1.$ Go to step 5 next.
    \item If we have maximal flow, stop. Otherwise, go to step 2 again.
\end{enumerate} 
\end{algo}
\par \cite{britz1999finite}[Theorem 7.3] says that the above algorithm maintains a minimum cost flow for each flow value at each step, comparing with the conditions in \cite{britz1999finite}[Theorem 7.1]. We will explicitly prove that this theorem holds even if we strip the potential bounding condition, for the sake of completeness.
\begin{theorem}[Modified Theorem 7.3, \cite{britz1999finite}]
The above algorithm produces, for each flow value, a minimal-cost flow, as the two conditions described in Theorem 3.2 are preserved after each step.
\par Furthermore, the algorithm terminates when we reach a maximal flow value.
\end{theorem}
\begin{proof}
We prove that the initial conditions have the desired properties in Theorem 3.2, and then that, after running through the algorithm, the desired properties hold, assuming that they held initially. This will prove the desired claim by induction, and hence Theorem 3.2. For ease of notation, let the index of $v$ be the value $i$ so that either $v = t_i$ or $v = b_i;$ initially, we see that the index of $v$ is just $-p(v)$ by construction.
\par First, for the initial conditions, notice that the flow everywhere is $0,$ by construction, so the first condition is vacuously true. As for the second, notice that $$p(w) - p(v) > c((v, w)) \implies f((v, w)) = u((v, w)),$$ means that the index of $w$ is smaller than that of $v.$ But then we have no edges from $w$ to $v,$ which means that this vacuously holds for all edges $(v, w) \in E.$
\par Now, for the algorithm. The only issues we need to check are for steps $3$ and $4.$ Suppose $G_{P, h, C_P}$ initially satisfied the conditions in Theorem 3.2. If we reach step 3, then by the algorithm we have a sequence of vertices $s, v_1, \ldots, v_k, t,$ where each consecutive pair of vertices in the sequence has an edge in $\bar{E},$ and we've increased the flow along these edges by $1.$ 
\par But notice that, by construction, the potentials between every pair of consecutive vertices equals the cost. This means that the conditions still hold, since the only pairs of vertices $(v, w)$ where the flow changes are those where $p(w) - p(v) = c((v, w)),$ so the conditions remain satisfied.
\par Now, suppose we reached step 4. Consider any edge $(v, w) \in E.$ If $p(w) - p(v) < c((v, w)),$ notice then that, since $p, c$ are always integers, $p(w) - p(v)$ remains at most $c((v, w)),$ and similarly for the $>$ symbol. The only thing we need to check is when $p(w) - p(v) = c((v, w))$ initially, and where exactly one of the potentials changes.
\par Suppose that $p(w)$ increases by $1.$ Then, it follows that $w$ is not in $X,$ but $v$ is in $X.$ But this means that, since we have a path from $s$ to $v$ along edges in $\bar{E},$ there is no edge between $v$ and $w$ in $\bar{E}.$ This means that, as $(v, w) \in E,$ we have that $f((v, w)) = u((v, w)),$ since the flow must remain at most the capacity. But then notice that this satisfies the condition.
\par Similarly, if $p(v)$ increases by $1,$ this means that $w \in X, v \not \in X.$ But again, this means that we have no edge from $w$ to $v.$ But this means that $f((v, w)) = 0,$ as $(v, w) \in E.$ This means that the condition is satisfied for that edge too. 
\par For maximality, we will prove this at the end of the next section. 
\end{proof}
\par This allows us to notice that, at every stage of the algorithm, even with a different potential function, we still output a minimal cost flow for a given flow value $v.$
\section{Basic Properties}
\par First, we prove some properties of the flow on $G_{P, h, C_P}$ in general, throughout the algorithm. We say that a vertex is ``reachable by $b_0$," or just ``reachable," if it lies in the set $X$ (as per the notation of \cite{britz1999finite}[\S 7], which we had for Algorithm 1). We first have the following lemma.
\par It's not hard to see that every vertex of the form $t_i$ or $b_i,$ where $1 \leq i \leq n,$ can have at most one edge with nonzero flow going in, and at most one edge with nonzero flow going out. To see this, notice that $t_i$ has only one edge that flows out, and $b_i$ has only one edge going into it, and all edges in this case have capacity $1.$ Since all flows are integral, by the algorithm, it follows that there can only be one edge for the other side that has nonzero flow. 
\par We now have two lemmas that we'd like to prove.
\begin{lemma}
For any edge from $b_i$ to $t_j,$ if there is a flow along that edge, then $p(t_j) - p(b_i)$ equals the cost of the edge.
\end{lemma}
\begin{proof}
Suppose for the sake of contradiction that this fails at some point during Algorithm 1. Consider the first step at which this fails, after making the change in flow or potential.
\par  Note that this can't be the first time that there is a flow between the two edges, since by construction we only add the flow if the cost equals the potential change. So this must mean that this occurs while potential drops; in other words, one of $b_i, t_j$ is reachable by $s$ along this new graph (in the sense that it lies in $X$) and the other isn't.
\par Suppose that $t_j$ is reachable by $b_0.$ Then, by construction, since before the change in potential the cost of flow along the edge equals the difference in potentials, we must have that $b_i$ is also reachable.
\par Similarly, if $b_i$ is reachable by $b_0,$ then there had to exist some point before it that allowed us to reach it. But this means that either we had to reach it via an unused edge (going forwards), or a used edge going backwards. The former, however, is impossible, since by the fact that there is flow out of $b_i$ there is flow into $b_i,$ and there is only one edge flowing into $b_i.$ \par This means that we had to have reached $t_j$ to get to $b_i.$ Hence, the supposed situation is impossible, which proves that the condition in the lemma always holds, as desired.
\end{proof}
\par In addition, we have the following property:
\begin{lemma}
For any $i \in \{1, 2, \ldots, n\},$ $p(t_i) \geq p(b_i) - 1.$
\end{lemma}
\begin{proof}
Again we proceed by contradiction. Suppose that at some point that $p(t_i) - p(b_i)$ was less than $-1,$ for some $i.$ Then, since $p(t_i) - p(b_i)$ can only increase or decrease by $1$ at each point, at some point, then, $p(t_i) - p(b_i) = -1.$ Furthermore, at this point, only $t_i$ was reachable by $b_0,$ and $t_{n+1}$ wasn't reachable, to cause the potential difference to change.
\par By the conditions given in Theorem 3.2, there has to be a flow from $t_i$ to $b_i.$ Then, note that, since there is flow into $t_i,$ there has to be another vertex, $b_k,$ with $k<i,$ where there is nonzero flow along the edge from $b_k$ to $t_i.$ If $k=0,$ then we can't reach $i$ directly from $s$ via an unused edge; this means that there is some other vertex $b_h$ with $h < i$ and where $p(t_i) - p(b_h) = c((b_h, t_i)).$ We take that vertex instead. Otherwise, if $k \neq 0,$ we just take $b_k.$ In either case, notice that we have that $p(t_i) - p(b_k) = c((b_k, t_i)),$ with the case $k \neq 0$ following from Lemma 4.1. 
\par In addition, consider the next vertex along the flow line, say $t_j,$ $j > i,$ after $b_i.$ Notice that there can't be any flow from $b_k$ to $t_j,$ as the edge from $b_i$ to $t_j$ has nonzero flow, and $k < i.$
\par We now do casework:
\begin{enumerate}
    \item $p(b_k) = p(t_i).$ Since the cost of an edge is either $-1$ or $0,$ we have that, by Lemma 4.1, $p(t_j) - p(b_i) = c((b_i, t_j)) \geq -1,$ or that $p(t_j) \geq p(t_i) = p(b_k).$ But this means that the cost of the edge between $b_k$ and $t_j$ is at most $p(t_j) - p(b_k)$. Since there can't be any flow between them, the cost must be at least the potential difference, so their potential difference is the same as the cost of the edge between them. However, since $b_k$ is reachable, this means that $t_j$ is too, which means that $b_i$ is reachable, contradiction.
    \item $p(b_k) = p(t_i) + 1.$ By a similar logic as above, we have that $p(t_j) \geq p(t_i) = p(b_k) - 1.$ But also, since there can't be nonzero flow in the edge between $b_k$ and $t_j,$ notice that $p(t_j) - p(b_k) \leq c((b_k, t_j)) \leq 0.$ Hence, either $p(b_k) = p(t_j),$ or $p(b_k) = p(t_j) + 1.$ The former gives us the same logic as the first case. For the latter, note that for this to occur, $p(t_j) = p(t_i) = p(b_i) - 1,$ or that $p(t_j) - p(b_i) = -1.$ But by Lemma 4.1, as we have flow on the edge from $b_i$ to $t_j$, this potential difference equals $c((b_i, t_j)).$ But this means that either $h^{-1}(k) < h^{-1}(i) < h^{-1}(j),$ or $t_j = t_{n+1}.$ In either case, note that this means that the cost of the edge between $b_h$ and $t_j$ is $-1$ and is equal to their potential difference, meaning that $t_j,$ and hence $b_i,$ is reachable. 
\end{enumerate}
In either case, we run into a contradiction, which proves the lemma.
\end{proof}
From here, we can now prove that we eventually get maximality from Theorem 3.3. 
\begin{proof}[Proof of Theorem 3.3, continued]Suppose for the sake of contradiction that this doesn't ever reach maximal flow. Then, Algorithm 1 doesn't terminate, and so eventually reaches a point where step 4 is constantly repeated, as step 3 increases flow and this maximal flow is well-defined; see the Ford-Fulkerson theorem, which is, for instance, \cite{williamson_2019}[Theorem 2.6]. 
\par In fact, here we can be more precise: notice that the maximal flow value is $n.$ To see this, notice that the value of the flow is the sum of the flows of the edges coming out of $b_0;$ with $n$ edges with capacity $1,$ this is at most $n.$ But to see maximality, notice that taking the edges between $b_0$ and $t_i,$ $t_i$ and $b_i,$ and $b_i$ to $t_{n+1},$ for each $i \in \{1, 2, \ldots, n\},$ gives a flow with value $n.$ Hence, maximal flow is $n.$ Therefore, for the sake of contradiction, we see that the flow value we reach is $v < n.$
\par Now, notice that, in general, step 4 cannot make $|X|$ fall; indeed, notice that step 4 alters potentials of vertices outside of $X,$ and doesn't alter flows, so every vertex in $X$ remains in $X.$
\par This means that, for us to never have $t \in X,$ eventually $X$ reaches some maximal set $X',$ since the number of elements is at most $2n+2.$ Furthermore, beyond this point, all of the flows of edges in $G_{P, h, C_P}$ remain constant. Consider the elements that must lie in this set $X'.$ 
\par Given that the only edges from $b_0$ are to vertices of the form $t_i,$ and that furthermore by construction in Algorithm 1 flows for each edge are integers (either $0$ or $1$), it follows that there is some $t_i,$ $i$ an integer between $1$ and $n,$ inclusive, so that the edge from $b_0$ to $t_i$ has flow $0$ (since we are assuming non-maximal flow). By the second part of Theorem 3.3, it follows that $p(t_i) - p(b_0) = p(t_i) \leq 1.$ If $t_i$ wasn't in $X'$ it would follow that the potential of $t_i$ would repeatedly increase by $1,$ contradicting this inequality.
\par From here, we have two cases. If the edge between $t_i$ and $b_i$ doesn't have a flow, then it follows that $p(b_i) - p(t_i) \leq 1,$ which using the above means that $p(b_i) - p(b_0) = p(b_i) - p(t_i) + p(t_i) - p(b_0) \leq 2.$ But again, by the same argument above, $b_i$ must lie in $X',$ as otherwise its potential will be unbounded as we continually repeat step 4 in Algorithm 1 (with $X$ never changing from $X'$). 
\par Now, notice that, since the only edge that points to $b_i$ is from $t_i,$ by construction, and since we assumed that the flow on the edge was $0,$ the edge between $b_i$ and $t_{n+1}$ has flow zero too. But the exact same argument shows that $t_{n+1} \in X',$ which contradicts the fact that we did step 4.
\par Otherwise, there is a flow on the edge between $t_i$ and $b_i.$ But this means that, by flow conservation, there exists an edge pointing into $t_i$ with flow, say from $b_j.$ But notice that all of the flow values are integers, and since the capacities are $1,$ this edge has flow $1.$ But notice then that the only edge pointing into $b_j$ is from $t_j,$ and it has capacity $1.$ This means that the edge from $b_j$ to $t_{n+1},$ by flow conservation, has flow $0,$ meaning that $p(t_{n+1}) - p(b_j) \leq 1.$ However, notice that, by Lemma 4.1, we have that $p(b_j) = p(t_i) - c((b_j, t_i)) \leq p(t_i) + 1$ the latter by construction of the costs. 
\par This means, however, that $p(t_{n+1}) - p(b_j) + p(b_j) \leq 2 + p(t_i) \leq 3,$ which again means that $p(t_{n+1})$ is bounded, so $t_{n+1}$ has to lie in $X',$ contradiction. This means that step 4 isn't used here, proving maximality, as desired.
\end{proof}
\section{Relating Graph and Poset Quantities}
\par We now take $G_{P, h, C_P}$ and relate it back to $A'_k$ and $D'_k.$ We begin by translating the poset quantities to the quantities on the graph, specifically flows and potentials, which \cite{britz1999finite}[\S 8] also does. However, the way these quantities are related to the poset quantities is somewhat different here compared to the corresponding version in \cite{britz1999finite}[\S 8].
\begin{proposition}
In $G_{P, h, C_P}$ given a fixed flow volume $v,$ the minimal cost of the flow is equal to $-A'_v.$
\end{proposition} 
\begin{proof}
To see this, we will first show that this is attainable. To do this, suppose that we have sequences $S_1, S_2, \ldots, S_v$ that give the value $A'_v.$ If one of the sequences consists of elements $s_1, s_2, \ldots, s_l$ we add the flow line going from $b_0$ to $t_{h(s_1)},$ then $t_{h(s_1)}$ to $b_{h(s_1)},$ then $b_{h(s_1)}$ to $t_{h(s_2)},$ and so forth, until $b_{h(s_l)}$ to $t_{n+1}.$ By construction, notice that we may do this, since the edges from $b_{h(s_i)}$ to $t_{h(s_{i+1})}$ exist by construction, as we demanded $(s_i, s_{i+1})$ to lie in $C_P$ for the sequences. 
\par Doing this for each sequence gives us the flow. Note that this satisfies the flow requirements, since at each vertex, the in and out flows are the same for all the vertices besides $b_0, t_{n+1}.$ In addition, we only use each edge once, since the vertices are all distinct in the sequences (from construction).
\par As for the cost of this flow, note that along each flow line, if it corresponds to sequence $S$ of elements $s_1, s_2, \ldots, s_l$ we see that all the edges have cost $0$ except those edges from $b_{h(s_i)}$ to $t_{h(s_{i+1})},$ where $s_i < s_{i+1}$ or from $b_{i_k}$ to $t_{n+1}.$ But this means that this flow line goes through edges whose total cost is just $-asc(S),$ for this sequence $S.$ Adding this up over all flow lines yields a flow with cost $-A'_v$ and volume of flow $v.$
\par To show this is minimal, suppose we have some other flow with value $v.$ Given an edge from $b_0,$ we can ``follow" this edge (since each vertex has either only one edge going in or one edge going out, except for $t_{n+1}$ or $b_0,$ and by conservation of flow there is exactly one for each) until we reach $t_{n+1}.$ This gives us a sequence of vertices.
\par We can repeat this for all the other edges from $b_0,$ yielding us $v$ distinct sequences. Note then that the cost of this flow is just the negative sum of the ascents over each sequence, which is at least $-A'_v,$ by the argument above. Notice that also by the fact that we followed edges that every pair of adjacent elements in a given sequence lie in $C_P,$ so these are actually adjacentable sequences.
\par We thus see that $-A'_v$ is the minimum cost of a flow with flow volume $v,$ as desired.
\end{proof}
\par Now, we introduce another quantity. Let $p= |p(t_{n+1})|.$ We say that $P_p$ is the number of $i \in \{1, 2, \ldots, n\}$ such that $p(t_i) = p(b_i) \in \{-p+1, -p+2, \ldots, 0\}.$
\begin{proposition}
At any point along Algorithm 1, $P_p + A'_v \geq n + vp.$ 
\end{proposition} 
\begin{proof}
This method is very similar to that of \cite{britz1999finite}[\S 8], in that we argue that, at every step along the algorithm, this inequality must continue to hold. We cannot jump immediately to equality yet, however; this will be the subject of Section 6.
\par Note that when the flow increases by $1,$ the only thing that changes is the cost of the flow, which decreases (by construction of this new flow line from the algorithm) by $p,$ since along this new flow-line, for each edge, the cost is equal to the difference in potential. 
\par If there are no flow lines, then note that if $t_i$ is reachable, so is $b_i,$ and if $t_i$ has potential $0,$ it is reachable, so again we have no problems here (potential drop of $t$ doesn't change $P_p'$)
\par Now, suppose that the potential of $t_{n+1}$ increases by $1.$ Consider a given flow-line, say reaching top and bottom pairs with indices $i_1, i_2, \ldots, i_k.$ Then, note that if $t_{i_j}$ is reachable by $b_0,$ so is $b_{i_{j-1}}.$
\par We can consider consecutive blocks of vertices reachable by $b_0$ along this flow line. Suppose we have a block running from $b_{i_c}$ to $t_{i_d}.$ Note that, among these, their potentials stay the same. Furthermore, note that this sequence cannot cause $P_p'$ to drop; the only place where one is no longer counted was if initially $t_{i_d}$ and $b_{i_d}$ had the same potentials. But note that, from un-reachability, $p(t_{i_c})$ increases by $1,$ which means that it matches up with $p(b_{i_c})$ now.
\par Hence, the only way for there to be a drop would be either if a pair of vertices had potential $0$ and went up to $1,$ or if there is a block that went directly to $t_{i_1}.$ But these are mutually distinct events, for a potential of $0$ going up to $1$ can only mean that $t_{i_1}$ and $b_{i_1}$ had potential $0$ and weren't reachable (if any other pair of vertices had potential $0,$ the top would be reachable).  
\par This means that, when $p$ rises by $1,$ $P_p'$ drops by at most $v.$ But this the establishes the inequality.
\end{proof}
\par Note that $D'_p \geq P'_p.$ To see this, we let the sequences be so that the $i$th sequence has the indices of those whose potentials of the top and bottom vertices are all $-i+1.$
\par This is a valid sequence for two reasons. First, for the actual non-increasing part, note that between the bottom vertex of one and the top vertex of the next, the cost can't be less than the potential difference, which is $0$ (either there is no flow, or there is flow, which means that this follows from Lemma 4.1). Hence, we see that this forms a non-increasing sequence.
\par Now, we claim that if the potential of $t_a, b_a$ are $i,$ and that for $t_c, b_c$ is $i-1,$ then $(h^{-1}(c), h^{-1}(a)) \not \in C_P.$ To see this, if not we would have an edge from $b_c$ to $t_a.$ But then notice that Lemma 4.1 and the condition 1 from Theorem 3.2 requires that $p(t_a) - p(b_c) \leq c((b_c, t_a)) \leq 0,$ contradiction. This means that the sequences can't be semi-overlapping, so this is a valid choice of sequences, giving a value of the sum of the $desc$ over these sequences as $P_p'.$
\par The main result, that $A'_v$ and $D'_p$ are conjugate in the sense we described, will follow in the next section.
\section{Establishing Equality}
This section follows \cite{britz1999finite}[\S 5] in concept, though the actual method of calculation is slightly different, due to different conditions on the ascending and non-ascending sequences.
\par We use the same idea of considering intersections, however. Suppose that we are given sequences $d_1, d_2, \ldots, d_p$ as the non-increasing sequences that meet the condition for $D'_p,$ and $a_1, a_2, \ldots, a_v$ for $A'_v.$ Notice that if the $d_i$ are contained in sequences that are not semi-overlapping, then the $d_i$ are not semi-overlapping either.
\par Fixing some $a_i,$ note that $a_i \cap d_1, a_i \cap d_2, \ldots, a_i \cap d_p$ (the subsequences of $a_i$ that are also part of $d_1, d_2, \ldots, d_p,$ respectively) are also not pairwise semi-overlapping, from construction.
\par In fact, notice that if element $x \in a_i \cap d_m$ and $y \in a_i \cap d_j$ are so that $h(x) < h(y),$ then notice that, by construction, $(x, y) \in C_P.$ But this means that all elements in $a_i \cap d_m$ occur before those in $a_i \cap d_j.$ 
\par Now, notice that for pair of consecutive elements within $a_i \cap d_j,$ say $x$ and $x',$ there exists a non-ascent in-between $x$ and $x'$ in $a_i,$ as otherwise $x < x',$ contradiction. Furthermore, in-between these elements, by the argument above, no other $a_i \cap d_k$ can have elements, meaning that each element of $a_i \cap d_j,$ letting $j$ vary, other than the last for each, corresponds uniquely to a non-ascent.
\par This means that we have $\sum\limits_{j=1}^{p} |a_i \cap d_j| \leq p + (des(a_i)),$ where $des(a_i)$ is the number of ``non-ascents," which by definition we can see satisfies $des(a_i) = |a_i| - asc(a_i).$ Note that this isn't $desc(a_i).$ 
\par This in turn yields that $\sum\limits_{j=1}^{p} |a_i \cap d_j| \leq p + (des(a_i)) \leq p + |a_i| - asc(a_i).$ But then we have that $\sum\limits_{i=1}^v \sum\limits_{j=1}^{p} |a_i \cap d_j| \leq vp + \sum\limits_{i=1}^v |a_i| - A'_v.$
\par But by PIE, since the $a_i$ are disjoint and the $d_j$ are disjoint, we have that $D'_p = |\bigcup_{j=1}^p d_j| = |\bigcup_{j=1}^p d_j \cup \bigcup_{i=1}^v a_i| - |\bigcup_{i=1}^v a_i| + \sum\limits_{i=1}^v \sum\limits_{j=1}^{p} |a_i \cap d_j| \leq n - |\bigcup_{i=1}^v a_i| + vp + \sum\limits_{i=1}^v |a_i| - A'_v = n + vp - A'_v.$
\par For equality, now note, for each pair $(p, v)$ that are reachable for $|p(t_{n+1})|$ and flow value, respectively, we have that $n + vp - A'_v \geq D'_p \geq P'_p \geq n + vp - A'_v.$ 
\par To get the desired conjugacy, the exact argument at the end of \cite{britz1999finite}[\S 8] allows us to finish. Specifically, we know now that $D'_p + A'_v = n + vp,$ where $p, v$ are values that are attained for $p(t_{n+1})$ and flow value, respectively, during Algorithm 1. We just need to check that we can apply the argument in that section here to all of the indices.
\par Now, notice that, by Theorem 3.3, the algorithm terminates when flow is maximal for the graph, which is when $v = n$ (taking, for each $i \in \{1, 2, \ldots, n\}$ a flow from $b_0$ to $t_i$ to $b_i$ to $t_{n+1}$). Furthermore, note that $v$ starts at $0.$ 
\par Therefore, notice that, at this ending point, we have flow value $n$ and some potential $p_0.$ When this occurs, notice that $A'_n + D'_{p_0} = n + np_0 \implies D'_{p_0} = np_0.$ But notice that, by construction, we see that $D'_{p_0} \leq D'_n = n$ meaning that $p_0 = 0$ or $p_0 = 1,$ so by a similar argument we see that the value of $|p(t_{n+1})|$ attains all values between $1$ and $n.$
\par Thus, for each $i \in \{1, 2, \ldots, n\}$ when flow value increases from $i-1$ to $i$ in the algorithm, $\lambda_i = A_i' - A_{i-1}' = p.$
\par Notice that we can show that $\lambda_i$ and $\mu_i$ are partitions, from the same logic as in \cite[\S 8]{britz1999finite}. This is because, as we perform this process, we have that $p$ is weakly decreasing, giving us that the $\lambda_i$ are weakly decreasing. As for the $\mu_i,$ notice that, by a similar logic, when the potential goes from $p$ to $p-1$, we have that $\mu_p = D_p' - D_{p-1}' = (n + vp - A_v') - (n + v(p-1) - A_v') = v.$ But then, observe that, throughout the process, $p$ falls and $v$ rises, so again the $\mu_i$ are also weakly decreasing if we start from $i = 1.$ This gives us that these are partitions.
\par This yields us the desired conjugacy of $\lambda_i = A'_i - A'_{i-1}$ and $\mu_i = D'_i - D'_{i-1},$ as desired, which proves Theorem 2.1.
\section{Corollaries}
\par Theorem 2.1 gives us both the localized Greene's theorem for permuation posets and the original Greene-Kleitman duality theorem. We prove each of these results using Theorem 2.1 in this section.
\begin{corollary}[Localized Greene's Theorem, Lemma 2.1 \cite{lewis2020scaling}]
Let $\sigma$ be a permutation on $n$ elements, $\{1, 2, \ldots, n\}.$ Then, with $A_k^*$ as the maximal sum of the ascents of $k$ disjoint sequences, and $D_k^*$ as the maximal sum of the longest descending subsequences in $k$ consecutive sequences (as we noted in the introduction, Section 1, which are defined as per \cite{Lewis2019}), if $\lambda_k = A_k^* - A_{k-1}^*$ and $\mu_k = D_k^* - D_{k-1}^*,$ then $\lambda_1 + \lambda_2 + \cdots$ and $\mu_1 + \mu_2 + \cdots$ form conjugate partitions of $n.$
\end{corollary}
\begin{proof}
\par Take the poset of $1, 2, \ldots, n$ with the natural ordering, and suppose that $h$ is the inverse of the permutation $\sigma,$ which is a bijection. Let $C_P$ just be the set $\{(x, y)|1 \leq x, y, \leq n, h(x) < h(y)\};$ in this case, $h-$ordering and adjacentable are the same. Apply Theorem 2.1, obtaining $A_k'$ and $D_k'.$ 
\par Then, notice that $A_k'$ is the same as $A_k^*$ since $asc$ is defined the same way. To see this, notice that any sequence $S,$ with elements $s_1, s_2, \ldots, s_l,$ where $\sigma(s_j) < \sigma(s_{j+1})$ for each index $j,$ can be thought of as a subsequence of elements from $\sigma(1), \sigma(2), \ldots, \sigma(n),$ as the above tells us that $\sigma^{-1}(s_1), \sigma^{-1}(s_2), \ldots, \sigma^{-1}(s_l)$ is a strictly increasing sequence. This means we may re-write the sequence as $\sigma(x_1), \sigma(x_2), \ldots, \sigma(x_l)$ for an increasing sequence $x_1, \ldots, x_l.$ But then $asc(S)$ is just the number of indices $j$ where $\sigma(x_j) < \sigma(x_{j+1})$ plus one (or $0$ if $S$ is empty), which matches. This means that $A_k',$ as the maximum of the sum of $asc$ of $k$ disjoint sequences, is the same as $A_k^*.$
\par As for $D_k',$ first notice that $desc$ is defined the same way as well, since the condition that $s_i \not < s_j$ for each $i < j,$ with the totally ordered set, just means that the sequence must be strictly decreasing. Now, suppose that we have sequences $S_1, S_2, \ldots, S_k$ that give the maximal value, such that no two are semi-overlapping. 
\par Now, since $C_P$ is just $h-$ordering, notice that for each pair of elements $x, y \in \{1, 2, \ldots, n\},$ either $(x, y) \in C_P$ or $(y, x) \in C_P.$ We may thus re-index the sequences so that $\forall i < j, \forall a \in S_i, b \in S_j, h(a) < h(b)$ (the semi-overlapping condition allows us to do this re-indexing).
\par From here, suppose that some element $x \in \{1, 2, \ldots, n\}$ not in any of the $S_i.$ Let $j$ be the largest index so that $\exists a \in S_j$ where $h(a) < h(x),$ and suppose that $a$ is chosen so that $h(a)$ is the maximum value of $\{h(b)|b \in S_j, h(b) < h(x)\}.$ We may then add $x$ to $S_j$ right after $a;$ by construction, this preserves all of the conditions of non semi-overlapping. Furthermore, notice that the $\sum_{i=1}^k desc(S_i)$ cannot decrease; indeed, we may take the same descending sequence within $S_j.$ By maximality, this value also can't increase.
\par We may thus assume that maximal $S_1, S_2, \ldots, S_k$ covers all of the elements in $\{1, 2, \ldots, n\}.$ But notice then that, as required in \cite{Lewis2019}, $S_1|S_2|\ldots|S_k$ is the sequence $h^{-1}(1), h^{-1}(2), \ldots,h^{-1}(n),$ or $\sigma(1), \sigma(2), \ldots, \sigma(n).$ This means that the value of $D_k',$ as defined here, is the same as $D_k^*$. This proves the desired.
\end{proof}
\begin{corollary}[Classical Greene-Kleitman Duality Theorem, Theorem 1.6 \cite{GREENE197669}] 
Given a poset $P,$ let $A_k$ be the maximal number of elements within $k$ disjoint chains, and $D_p$ the maximal number of elements within $p$ disjoint anti-chains. Then, if $\lambda_i = A_i - A_{i-1}$ and $\mu_i = D_i - D_{i-1}$ for $i \geq 1,$ with $A_0 = D_0 = 0,$ then $\lambda_1 + \lambda_2 + \ldots$ and $\mu_1 + \mu_2 + \cdots$ are conjugate partitions of $n.$
\end{corollary}
\begin{proof}
Let $P$ be the poset, and $h$ any linear extension of $P.$ From here, let $C_P$ be just the set $\{(x, y)| x < y\};$ notice that this satisfies the properties given.
\par Then, notice that any adjacentable sequence, by construction, must consist solely of elements where any two adjacent are increasing; in other words, they must be chains. Therefore, it follows that $A'_k$ in Theorem 2.1 just corresponds to the maximal length of $k$ disjoint chains, which is just $A_k.$
\par As for $D_p',$ we need to do a little more work. Notice that $D_p' \leq D_p.$ To see this, suppose that sequences $S_1, \ldots, S_p$ had subsequences $d_1, \ldots, d_p,$ whose sum of lengths was $D_p'.$ By construction, for each sequence $d_j,$ if the elements in order were $s_{1, j}, \ldots, s_{l_j, j},$ then notice that the condition that $s_{a, j} \not < s_{b, j}$ for each $a, b,$ combined with the ordering $h,$ thus demands that, in fact, $s_{a, j}$ and $s_{b, j}$ are not comparable. This means that each of the $d_i$ are anti-chains.
\par To show the other direction: suppose that we have $p$ anti-chains by $d_1, d_2, \ldots, d_p$ so that their sum has maximal size. Consider the ordered tuple obtained by taking the elements for $d_1$ in order, followed by the elements for $d_2$ in order, and so forth, and order these lexicographically using the linear extension. For instance, if we have the poset on five elements $a, b, c, d, e,$ with relations $a < b,  b < d, c < d,$ and $d<e,$ with $h(a) = 1, h(b) = 2, h(c) = 3, h(d) = 4,$ and $h(e) = 5,$ taking $d_1$ to be the sequence $a, c$ and $d_2$ to be $b$ yields the tuple $(a, c, b).$
\par Now, consider the following operation: given $d_i$ and $d_j,$ where $i < j,$ let $A = \{x \in d_i|\exists y \in d_j \text{ so that } y<x\}.$ Similarly, let $B = \{y \in d_j|\exists x \in d_i \text{ so that } y<x\}.$ Then, take the elements from $A,$ and move them to $d_j,$ and take the elements from $A,$ and move them to $d_i.$ Call these new anti-chains $d_i', d_j'.$
\par First, note that the new $d_i$ and $d_j$ are both anti-chains. Suppose for the sake of contradiction this wasn't the case; then, since $d_i, d_j$ were anti-chains, the relations that occur afterwards must have one element in one of the sets $A, B$ and the other not (since, by anti-chain, all the elements in $A$ are pairwise incomparable, and similarly for $B$). This yields four cases:
\begin{enumerate}
    \item If there exists an $a \in d_i', b \in B$ so that $b  < a,$ then $a \in d_i'$ means that $a \not \in A.$ But $a \not \in B,$ so $a \in d_i,$ and $a \in A,$ contradiction.
    \item If there exists an $a \in d_i', b \in B$ so $a < b,$ then there exists an element $x$ in $d_i$ so that $b<x,$ so then $a < x.$ But $a \not \in B,$ so $a \in d_i,$ contradicting anti-chain.
    \item If there exists an $a \in A, b \in d_j
    '$ so that $b < a,$ then notice that $b \in d_j'$ means that $b \not \in B.$ But $a \in A \subseteq d_i,$ meaning that $b \in B,$ contradiction. 
    \item If there exists an $a \in A, b \in d_j$ so $a < b,$ then there exists a $y \in d_j$ so that $y < a < b,$ or $y < b.$ But $b \not \in A,$ so thus $b \in d_j,$ contradicting anti-chain.
\end{enumerate}
\par Therefore, we end up still with anti-chains, the sum of whose lengths is the same.
\par Furthermore, notice that the result we get is an element that is lexicographically earlier; let $x$ be so that $h(x)$ is minimal, among all elements of $A, B.$ Then, notice that, by construction, $x \in B,$ otherwise we see that there is a $y \in d_j$ so that $h(y) < h(x),$ meaning that $y \in B$ as $x \in A \subseteq d_i,$ contradicting minimality. Then, notice that this moves from the list of $j$s to the list of $i$s, and by construction no other elements are moved other than those in $A$ or $B.$ But $i < j$ means that this means it is lexicographically earlier.
\par Since we only have a finite number of these tuples, we can only apply this process a finite number of times before we end up with a result where, for any $i, j,$ the resulting $A, B$ are empty. But if $A, B$ are empty, notice then that these anti-chains are all not semi-overlapping, since the semi-overlapping condition for $d_i, d_j$ here requires that, for $i < j,$ that there exists $x \in d_i, y \in d_j$ so $y < x,$ or that the resulting $A, B$ aren't empty.
\par Therefore, we see that we can re-arrange the anti-chains in a way so that they are not semi-overlapping, so $D'_p \geq D_p \geq D'_p,$ and these are equal.
\par But this means that the conjugate partitions in this theorem are precisely those given in Theorem 2.1, as desired.
\end{proof}
\par Note that Example 2.1 yields a case that doesn't fall under either of these corollaries. In particular, we can view Corollary 7.1, the localized Greene's theorem, as being the case when $C_P$ is as large as possible, and poset $P$ is just $\{1, 2, \ldots, n\}.$ On the other hand, Corollary 7.2 occurs when $C_P$ is as small as possible, and $h$ is a linear extension.
\printbibliography
\typeout{get arXiv to do 4 passes: Label(s) may have changed. Rerun}
\end{document}